\documentclass[11pt,a4paper]{article}
\usepackage[latin1]{inputenc}
\usepackage{amsmath,amsthm,amsfonts,amssymb,amscd,latexsym}
\usepackage{enumerate}
\usepackage{epsfig,psfrag}
\usepackage{graphics,graphicx,bezier,float,color,hyperref}
\usepackage{url}
\usepackage{multienum}
\usepackage[table]{xcolor}
\usepackage{multicol,multirow}
\usepackage{fancyvrb}
\usepackage{tcolorbox}
\usepackage{parskip}
\usepackage[toc,page]{appendix}
\usepackage{booktabs}
\usepackage[top=2.7cm, bottom=2.7cm, left=1.5cm, right=1.5cm]{geometry}
\usepackage[none]{hyphenat}
\usepackage{blkarray}
\newtheorem{theorem}{Theorem}[section]
\newtheorem{lemma}[theorem]{Lemma}

\newtheorem{definition}[theorem]{Definition}

\numberwithin{equation}{section}
\numberwithin{table}{section}
\numberwithin{figure}{section}

\title{On the Diophantine Inequality $\lvert x^{2} - 2^{a}\cdot 3^{b}\rvert < 3\max\{a,b\}$}
\author{Banu \.{I}rez Ayd{\i}n$^{1}$ \and Herbert Batte$^{2}$ \and \.{I}lker \.{I}nam$^{3}$ \and  Florian Luca$^{2,4}$ \and Zeynep Demirkol \"{O}zkaya$^{5}$}
\date{}

\begin{document}
\maketitle
\abstract{ In this paper, we show that there are $57$ nonnegative integer solutions $(a,b,x)$ to the inequality $1\le \lvert x^{2} - 2^{a}\cdot 3^{b}\rvert < 3\max\{a, b\}$ and we list  them explicitly. The inequality is converted into a statement about how closely $x/q$ approximates irrational number $\sqrt{d}$ for $d\in\{2,3,6\}$, where $q$ is an integer which is $3$-smooth, after which Worley's theorem on rational approximations via continued fractions is applied to parametrise the solutions and a $p$-adic lower bound for a linear form in logarithms due to Bugeaud and Laurent is applied to find a rather large bound on $\max\{a,b\}$. We finish with an application of the ${\text{\rm LLL}}$ algorithm to reduce this bound.} 

\medskip


\medskip


\medskip


\section{Introduction}\label{intro}
\subsection{Background}
A positive integer is called \emph{$B$-smooth} if all of its prime factors are at most $B$. Thus, numbers whose only prime factors belong to the set $\{2,3\}$ are called $3$-\emph{smooth} numbers \footnote{see A003586 at oeis.org}. They form the sequence
\begin{align*}
1,\ 2,\ 3,\ 4,\ 6,\ 8,\ 9,\ 12,\ 16,\ 18,\ 24,\ \ldots ,
\end{align*}
and can be written explicitly as $\{2^{a}3^{b} : a,b\geq 0\}$.  These numbers arise naturally in many branches of mathematics, for example, in harmonic analysis (the Fourier structure of $\mathbb{Z}/n\mathbb{Z}$ is most transparent when $n$ is 3-smooth), in algorithmic Number theory (FFT
algorithms run optimally on 3-smooth sizes), and in classical Diophantine Number theory, where the interaction between the additive structure of perfect powers and the multiplicative structure of smooth numbers generates a rich family of problems.

A central question in this latter direction is: \emph{how closely can a perfect square approximate a smooth number?} The equation
\begin{align}\label{eq:elem}
x^{2} = 2^{a}3^{b},
\end{align}
is elementary. It forces both $a$ and $b$ to be even, yielding the infinite family of solutions $x = 2^{a/2}3^{b/2}$. The problem becomes significantly more delicate when one asks instead for \emph{near-misses},
that is, for integer triples $(x,a,b)$ for which $x^{2}$ and $2^{a}3^{b}$ are close but not equal. This is the perspective of Diophantine inequalities.

A fundamental motivation for such problems comes from the theory of Pillai-type equations. Pillai \cite{Pillai1936} studied the equation $a^{x} - b^{y} = c$ in integers, and the systematic study of how close two perfect powers can be to each other has since generated a large body of literature. We refer the reader to Shorey and Tijdeman in \cite{ShoreyTijdeman1986} for a comprehensive treatment. 
Closely related to our problem is the family of equations of the form
\begin{align*}
x^{2} + 2^{a}3^{b} = y^{n}\qquad {\text{\rm with}}\qquad \gcd(x,y)=1.
\end{align*}
Luca \cite{Luca2002}, determined all solutions when $n\geq 2$, while Luca and Togb\'e \cite{LucaTogbe2008} treated the analogous problem with the prime $3$ replaced by $5$. In a further generalization, the authors of \cite{cangul2013diophantine} completely solved
\begin{align*}
x^{2} + 2^{a}3^{b}11^{c} = y^{n}\qquad  {\text{\rm with}}\qquad \gcd(x,y)=1,
\end{align*}
showing that the presence of an additional prime factor significantly increases the technical complexity. Similar work in this line can also be found in \cite{gica2012diophantine}. 

All of these works concern Diophantine equations. The present paper investigates instead a \emph{Diophantine inequality}, asking how closely a perfect square can approximate a 3-smooth number when an error
of moderate size is permitted. Specifically, we investigate the Diophantine inequality
\begin{equation}\label{eq:ineq}
	1\le \lvert x^{2} - 2^{a}\cdot 3^{b}\rvert < 3\max\{a,b\},
\end{equation}
in nonnegative integers $x$, $a$, $b$. The right-hand side grows only linearly in the exponents $a$ and $b$, while $2^{a}3^{b}$ grows exponentially. The problem therefore lies in a highly constrained regime and one expects only finitely many solutions.
 Problems of this type contribute to the broader program of understanding Diophantine inequalities in which polynomial expressions are close to multiplicatively structured
sets, and they further demonstrate the effectiveness of transcendence methods in completely resolving such questions.
Our main result is the following.
\subsection{Main Result}
\begin{theorem}\label{thm:main}
	There are exactly $57$ nonnegative integer solutions $(a, b, x)$ satisfying \eqref{eq:ineq}, and are listed in Table \ref{tab:small}.
\end{theorem}

\begin{table}[H]
	\centering
	\label{tab:small}
	\renewcommand{\arraystretch}{1.25}
	\begin{tabular}{ccc|ccc|ccc}
		\toprule
		$a$ & $b$ & $x$ & $a$ & $b$ & $x$ & $a$ & $b$ & $x$ \\
		\midrule
		0  &  1 &    1 &  2 &  1 &    3 &  5 &  1 &   10 \\
		0  &  1 &    2 &  2 &  1 &    4 &  5 &  2 &    17 \\
		0  &  2 &    2 &  2 &  3 &    10 &  5 &  4 &    51 \\
		0  &  3 &    5 &  2 &  5 &   31 &  6 &  0 &   7 \\
		0  &  5 &   16 &  3 &  0 &   0 &  6 &  0 &   9 \\
		0  & 15 & 3788 &  3 &  0 &    1 &  6 &  1 &   14 \\
		1  &  0 &    0 &  3 &  0 &    2 &  7 &  0 &    11 \\
		1  &  0 &    1 &  3 &  0 &    3 &  7 &  0 &    12 \\
		1  &  0 &    2 &  3 &  0 &    4 &  7 &  1 &   20 \\
		1  &  1 &    2 &  3 &  1 &    4 &  7 &  2 &   34 \\
		1  &  2 &    4 &  3 &  1 &    5 &  8 &  1 &   28 \\
		1  &  3 &    7 &  3 &  2 &    8 &  8 &  3 &   83 \\
		1  &  4 &   13 &  3 &  5 &    44 &  9 &  0 &   23 \\
		1  &  5 &   22 &  4 &  0 &   3 &  9 &  1 &   39 \\
		1  &  6 &   38 &  4 &  0 &    5 &  9 &  2 &   68 \\
		1  &  7 &   66 &  4 &  1 &    7 &  11 &  0 &   45 \\
		2  &  0 &    0 &  4 &  3 &    21 &  12 &  1 &   111 \\
		2  &  0 &    1 &  5&  0&  5&  15 &  0 &   181 \\
		2  &  0 &    3 &  5&  0&  6& 17 &  0 &   362 \\
		\bottomrule
	\end{tabular}	
	\caption{All triples $(a, b, x)$ satisfying $1 \leq \lvert x^2 - 2^a\cdot 3^b\rvert < 3\max\{a,b\}$.}
\end{table}

\subsection{Proof strategy}

The proof splits naturally into two parts according to the size of $A := \max\{a,b\}$.

\medskip

\noindent\textbf{Small cases ($A \leq 2200$).} A computation in \textsc{SageMath} searches over all pairs $(a,b)$ of integers with $a,b \in [0,2200]$ and locates all solutions to
the inequality, yielding exactly $57$ triples $(a,b,x)$
(Table \ref{tab:small}). If one wants solutions where all $x,a,b$ are positive, then there are exactly $28$ of these in the above table. 

\medskip
\noindent\textbf{Large cases ($A > 2200$).}
The case where both $a$, $b$ are even forces $x = 2^{a/2}3^{b/2}$, so $|x^2 - 2^a 3^b| = 0 < 1$, so no solutions arise. For the remaining three parity classes we write $2^a 3^b = 2^{2\alpha}3^{2\beta}d$ with $d \in \{2,3,6\}$, and observe that any solution to the inequality gives a rational approximation $x_0/y_0$ to ${\sqrt{d}}$ with small error. Worley's theorem \cite{Worley1981} characterizes all such approximations in terms of the convergents of $\sqrt{d}$, yielding a unified equation
\begin{equation*}
	2\sqrt{d}\,y_0 = C\gamma^n + D\delta^n,
\end{equation*}
where $\gamma$ is the fundamental unit of $\mathbb{Q}(\sqrt{d})$ and $(C,D)$ are algebraic integers parametrized by five explicit families. Taking $p$-adic valuations
where $p\in \{2,3\}$ of both sides of the above expression and applying the result of Bugeaud and
Laurent \cite{BL} on $p$-adic linear forms in two logarithms gives the bound $A < 3.57 \times 10^7$. The LLL algorithm then
reduces this bound to $A \leq 2200$, bringing all remaining solutions into the range already handled.

A special  case arises, namely when the two algebraic numbers for which one intends to apply the result of Bugeaud--Laurent are multiplicatively dependent.
In this case, an elementary argument shows that in fact $A$ is much smaller ($A < 200$), so all solutions are again captured by the small-cases computation.

\medskip

The paper is organised as follows. Section \ref{sec:prelim} collects the necessary background: $p$-adic valuations and the Bugeaud--Laurent theorem, an exponent--lifting lemma from \cite{post}, continued fractions with Legendre's theorem and Worley's
theorem, and the LLL algorithm. Section \ref{sec:proof} contains the proof of Theorem \ref{thm:main}.

\section{Preliminaries}\label{sec:prelim}

\subsection{Linear forms in logarithms and $p$-adic valuations}

Before formulating bounds on linear forms in logarithms, we need the notion of height of an algebraic number and a few of its properties.

\begin{definition}\label{def1}
	Let $ \eta $ be an algebraic number of degree $ d $ with minimal primitive polynomial over the integers given by
	$$ a_{0}x^{d}+a_{1}x^{d-1}+\cdots+a_{d}=a_{0}\prod_{i=1}^{d}(x-\eta^{(i)}), $$ 
	where the leading coefficient $ a_{0} $ is positive. Then, the logarithmic height of $ \eta$ is given by 
	$$ h(\eta):= \dfrac{1}{d}\Big(\log a_{0}+\sum_{i=1}^{d}\log \max\{|\eta^{(i)}|,1\} \Big). $$
\end{definition}
In particular, if $ \eta$ is a rational number represented as $\eta:=a/b$ with coprime integers $a$ and $ b\ge 1$, then $ h(\eta ) = \log \max\{|a|, b\} $. 
The following properties of the logarithmic height function $ h(\cdot) $ will be used in the rest of the paper without further reference:
\begin{equation}\nonumber
	\begin{aligned}
		h(\eta_{1}\pm\eta_{2}) &\leq h(\eta_{1})+h(\eta_{2})+\log 2;\\
		h(\eta_{1}\eta_{2}^{\pm 1} ) &\leq h(\eta_{1})+h(\eta_{2});\\
		h(\eta^{s}) &= |s|h(\eta)  \quad {\text{\rm valid for}}\quad s\in \mathbb{Z}.
	\end{aligned}
\end{equation}

We now define $p$-adic valuations, where $p$ is a prime number.

\begin{definition}\label{def2.2t}
	Let \( p \) be a prime number. The \( p \)-adic valuation of an integer \( x \), denoted by $\nu_p(x)$,  is defined by
	\[
	\nu_p(x) := 
	\begin{cases} 
		\max\{k \in \mathbb{N} : p^k \mid x\}, & \text{if } x \neq 0; \\
		\infty, & \text{if } x = 0.
	\end{cases}
	\]
	Additionally, if $x=a/b$ is a rational number and $a,~b$ are integers, then we put 
	$$\nu_p(x):=\nu_p(a)-\nu_p(b).
	$$
\end{definition}

The formula for $\nu_p(x)$ when $x$ is rational given by Definition \ref{def2.2t} does not depend on the representation of $x$ as a ratio of integers $a/b$. It follows easily from Definition \ref{def2.2t} that if $x$ is rational, then
$$
\nu_p(x)=\text{ord}_p(x),
$$
where $\text{ord}_p(x)$ is the exponent of $p$ in the factorization of $x$. For example, $\nu_2(9/8)=-3$. Now for the algebraic number $\eta$ in Definition \ref{def1}, we define
\[
\nu_p(\eta):=\frac{\nu_p(a_d/a_0)}{d},
\]
that is, it is the $p$--adic valuation of the rational number $a_d/a_0$ divided by the degree $d$ of $\eta$. For example, when $x=a/b$ is a rational number with coprime integers $a$ and $b\ge 1$, then 
its minimal polynomial is $f(X)=bX-a$ has degree $1$ so 
$$
\nu_p(x)=\nu_p(a/b),
$$
consistent with Definition \ref{def2.2t}.

Next, let $\eta_{1}$ and $\eta_{2}$ be algebraic numbers over $\mathbb{Q}$. Let ${\mathbb K}:={\mathbb Q}(\alpha_1,\alpha_2)$ and let $D_{\mathbb K}:=[{\mathbb K}:{\mathbb Q}]$ be the degree 
of ${\mathbb K}$. Consider the algebraic number $\eta_1^{b_1}\eta_2^{b_2}-1$ and assume that it is not zero. Then one might want a good upper bound on 
$\nu_p(\eta_1^{b_1}\eta_2^{b_2}-1)$. Let 
\[h'(\eta_{i}) \ge \max \left\{h(\eta_{i}),\dfrac{\log p}{D_{\mathbb K}}\right\}, \qquad\text{for }~i=1,2.\]
The following result is Corollary 2 in \cite{BL}. 
\begin{theorem}[Bugeaud and Laurent, \cite{BL}]\label{thm:Bugt}
	Let $b_1$, $b_2$ be positive integers and suppose that $\eta_{1}$ and $\eta_{2}$ are multiplicatively independent algebraic numbers such that $\nu_p(\eta_{1})=\nu_p(\eta_{2})=0$. Put 
	$$  E':=\dfrac{b_1}{h'(\eta_{2})}+\dfrac{b_2}{h'(\eta_{1})},   $$
	and 
	$$ E:=\max\left\{\log E'+\log\log p+0.4, 10, 10\log p\right\}.    $$
	Then
	$$ \nu_p\left(\eta_{1}^{b_1}\eta_{2}^{b_2}-1\right)\le \frac{24pg}{(p-1)(\log p)^4} E^2D_{\mathbb K}^4   h'(\eta_{1})h'(\eta_{2}  ), $$
	where $g>0$ is the smallest integer such that $\nu_p\left(\eta_{i}^g-1\right)>0$ for both $i=1,2$.
\end{theorem}

When $\eta_1$ and $\eta_2$ are multiplicatively dependent, then $\eta_1^{b_1}\eta_2^{b_2}-1$ becomes associated with $\eta^t\zeta-1$, where $\zeta$ is some root of unity in 
${\mathbb K}$, $\eta$ is a generator of a torsion free part of the multiplicative subgroup of ${\mathbb K}$ generated by $\eta_1$ and $\eta_2$ and $t$ is some integer. For us, ${\mathbb K}$ 
is real quadratic so $\zeta\in \{\pm 1\}$. The following result is originally from \cite{post} and appears also as Lemma 4.4 in \cite{BiLu}. For a number field ${\mathbb K}$, a prime ideal 
$\pi$ of ${\mathbb O}_{\mathbb K}$ and a nonzero element $\eta\in {\mathbb K}$, we write $\nu_{\pi}(\eta)$ for the exponent of $\pi$ in the factorization of the fractional ideal $\eta{\mathcal O}_{\mathbb K}$ of 
of ${\mathbb K}$. The connection between $\nu_{\pi}$ for prime ideals $\pi$ of ${\mathcal O}_{\mathbb K}$ and the $p$-adic valuation $\nu_p$ is the following. Let $p\in {\mathbb Z}$ be prime and
$$
p{\mathcal O}_{\mathbb K}=\pi_1^{e_1}\cdots \pi_k^{e_k},
$$
where $\pi_1,\ldots,\pi_k$  are distinct prime ideals in ${\mathcal O}_{\mathbb K}$. The number $e_i$ is called the ramification index of the prime ideal $\pi_i$ in ${\mathbb K}$ for $i=1,\ldots,k$.  Then 
\begin{equation}
\label{eq:nup}
\nu_p(\eta)=\frac{1}{D_{\mathbb K}}\left(\sum_{i=1}^k f_i\nu_{\pi_i}(\eta)\right),
\end{equation}
where in the above, $f_i$ is dimension of ${\mathcal O}_{\mathbb K}/\pi_i$ as an ${\mathbb Z}/p{\mathbb Z}$ vector space; namely the cardinality of the field ${\mathcal O}_{\mathbb K}/\pi_i$ is $p^{f_i}$ for $i=1,\ldots,k$.

\begin{lemma}
\label{lem:BiLu}
	Let $\mathbb{K}$ be a number field, $\pi$ be a prime ideal 
	of $\mathcal{O}_{\mathbb{K}}$ above the prime $p\in {\mathbb Z}$ of
	index $e$, and let $\zeta \in \mathcal{O}_{\mathbb{K}}$. Suppose 
	$$\nu_{\pi}(\zeta - 1) > \frac{e}{p - 1}.
	$$
	Then for any nonzero integer $m$,
	$$
	\nu_{\pi}(\zeta^m - 1) = \nu_{\pi}(\zeta - 1) + \nu_{\pi}(m).
	$$
\end{lemma}

\subsection{Continued fractions}

Recall that every irrational number $\tau$ has a unique representation as an \emph{infinite simple continued fraction} $\tau = [a_0; a_1, a_2, \ldots]$ with partial quotients $a_i \in \mathbb{Z}_{\geq 0}$, $a_i \geq 1$ for
$i \geq 1$. The \emph{convergents} are given by $p_n/q_n: = [a_0; a_1, \ldots, a_n]$ and are the best rational approximations to $\tau$ in the sense that
\begin{equation*}
	\left|\tau - \frac{p_n}{q_n}\right| < \frac{1}{q_n q_{n+1}}.
\end{equation*}
The denominators satisfy the three-term recurrence
$q_{n+1} = a_{n+1}q_n + q_{n-1}$ for $n\ge 1$ with initial values $q_0=0$, $q_1=a_1$.

Worley's theorem \cite{Worley1981} is a converse to the classical theory of continued fractions. It characterizes all rationals  $p/q$ that approximate $\tau$ to within $K/q^2$, not just the convergents.

\begin{theorem}[Worley \cite{Worley1981}]\label{thm:Worley}
	Let $\tau$ be a positive irrational number, $K>0$, 
	and let $p/q$ be a rational approximation to $\tau$ in reduced form 
	satisfying
	\begin{equation*}
		q^{2}\left|\tau - \frac{p}{q}\right| < K.
	\end{equation*}
	Then either $p/q$ is a convergent $p_n/q_n$ to $\tau$, or $p/q$ has one of the following forms for some index $n \geq 1$ and integers $r\ge 0, s\ge 1$:
	\begin{align*}
		\text{\rm(i)}\quad  &\frac{p}{q} = \frac{rp_n + sp_{n-1}}{rq_n + sq_{n-1}},
		&rs &< 2K;\\
		\text{\rm(ii)}\quad &\frac{p}{q} = \frac{rp_n - sp_{n-1}}{rq_n - sq_{n-1}},
		&rs &< 2K;\\
		\text{\rm(iii)}\quad &\frac{p}{q} = \frac{rp_{n+1} + sp_{n-1}}{rq_{n+1} 
			+ sq_{n-1}}, &rs &< K,\quad a_{n+1} = 1.
	\end{align*}
\end{theorem}

When $K<1/2$ above, Worley's theorem gives that $rs<1$, so only the case $s=1$ is possible. This gives that $p/q=p_{n-1}/q_{n-1}$ is a convergent of $\tau$, a result due to Legendre. 
We also need the following result (see Theorem~8.2.4 in \cite{ME}).

\begin{lemma}[Legendre]\label{lem:Legendre}
	Let $\xi\in {\mathbb R}\backslash {\mathbb Q}$ with continued fraction expansion
	$\xi = [a_0; a_1, a_2, \ldots]$ and let again
	$p_i/q_i = [a_0; a_1, \ldots, a_i]$ be its $i$-th convergent for $i\ge 0$.  Let $M$ be a positive integer and let $n$ be a nonnegative integer such that $q_n> M$. Put
	$$
	a(M) := \max\{a_i : i = 0, 1, 2, \ldots, M\}.
	$$
	Then the inequality
	$$
	\left|\xi - \frac{f}{w}\right| > \frac{1}{(a(M)+2)\,w^2}
	$$
	holds for all pairs $(f, w)$ of positive integers with $0 < w < M$.
\end{lemma}
We record Lemma \ref{lem:Legendre} here for context, as it is the classical quantitative form of the statement that the best rational approximations to an irrational number arise from its continued fraction convergents. In the proof of Theorem \ref{thm:main}, the relevant consequence of this principle (namely, that the case $K<1/2$ in Theorem \ref{thm:Worley} forces $rs<1$, so that $x_0/y_0$ is itself a convergent of $\sqrt{d}$) is obtained directly as a
corollary of Worley's theorem, and the resulting case is treated via the multiplicative-dependence argument in Section \ref{sec:proof}.

\subsection{The LLL algorithm}

Typically, the estimates from Theorem \ref{thm:Bugt} are excessively large to be practical in computations. To refine these estimates, we use a method based on the LLL-algorithm. While Lemma \ref{lem:Legendre} is often sufficient for linear forms in two, the LLL-algorithm offers a more robust and generalized approach for three or more logarithms, as is the case in our study (see \cite{SMA, Weg}). We next explain this method.

Let $k$ be a positive integer. A subset $\mathcal{L}$ of the $k$-dimensional real vector space ${ \mathbb{R}^k}$ is called a lattice if there exists a basis $\{b_1, b_2, \ldots, b_k \}$ of $\mathbb{R}^k$ such that
\begin{align*}
	\mathcal{L} = \sum_{i=1}^{k} \mathbb{Z} b_i = \left\{ \sum_{i=1}^{k} r_i b_i \mid r_i \in \mathbb{Z} \right\}.
\end{align*}
We say that $b_1, b_2, \ldots, b_k$ form a basis for $\mathcal{L}$, or that they span $\mathcal{L}$. We
call $k$ the rank of $ \mathcal{L}$. The determinant $\text{det}(\mathcal{L})$, of $\mathcal{L}$ is defined by
\begin{align*}
	\text{det}(\mathcal{L}) = | \det(b_1, b_2, \ldots, b_k) |,
\end{align*}
with the $b_i$'s being written as column vectors. This is a positive real number that does not depend on the choice of the basis (see \cite{Cas}, Section 1.2).

Given linearly independent vectors $b_1, b_2, \ldots, b_k $ in $ \mathbb{R}^k$, we refer back to the Gram--Schmidt orthogonalization technique. This method allows us to inductively define vectors $b^*_i$ (with $1 \leq i \leq k$) and real coefficients $\mu_{i,j}$ (for $1 \leq j \leq i \leq k$). Specifically,
\begin{align*}
	b^*_i &= b_i - \sum_{j=1}^{i-1} \mu_{i,j} b^*_j,~~~
	\mu_{i,j} = \dfrac{\langle b_i, b^*_j\rangle }{\langle b^*_j, b^*_j\rangle},
\end{align*}
where \( \langle \cdot , \cdot \rangle \)  denotes the ordinary inner product on \( \mathbb{R}^k \). Notice that \( b^*_i \) is the orthogonal projection of \( b_i \) on the orthogonal complement of the span of \( b_1, \ldots, b_{i-1} \), and that \( \mathbb{R}b_i \) is orthogonal to the span of \( b^*_1, \ldots, b^*_{i-1} \) for \( 1 \leq i \leq k \). It follows that \( b^*_1, b^*_2, \ldots, b^*_k \) is an orthogonal basis of \( \mathbb{R}^k \). 
\begin{definition}
	The basis $b_1, b_2, \ldots, b_n$ for the lattice $\mathcal{L}$ is called reduced if
	\begin{align*}
		\| \mu_{i,j} \| &\leq \frac{1}{2}, \quad \text{for} \quad 1 \leq j < i \leq n,~~
		\text{and}\\
		\|b^*_{i}+\mu_{i,i-1} b^*_{i-1}\|^2 &\geq \frac{3}{4}\|b^*_{i-1}\|^2, \quad \text{for} \quad 1 < i \leq n,
	\end{align*}
	where $ \| \cdot \| $ denotes the ordinary Euclidean length. The constant $ {3}/{4}$ above is arbitrarily chosen, and may be replaced by any fixed real number $ y $ in the interval ${1}/{4} < y < 1$ {\rm(see \cite{LLL}, Section 1)}.
\end{definition}
Let $\mathcal{L}\subseteq\mathbb{R}^k$ be a $k-$dimensional lattice  with reduced basis $b_1,\ldots,b_k$ and denote by $B$ the matrix with columns $b_1,\ldots,b_k$. 
We define
\[
l\left( \mathcal{L},y\right)= \left\{ \begin{array}{c}
	\min_{x\in \mathcal{L}}||x-y|| \quad  ;~~ y\not\in \mathcal{L}\\
	\min_{0\ne x\in \mathcal{L}}||x|| \quad  ;~~ y\in \mathcal{L}
\end{array}
\right.,
\]
where $||\cdot||$ denotes the Euclidean norm on $\mathbb{R}^k$. It is well known that, by applying the
LLL--algorithm, it is possible to give in polynomial time a lower bound for $l\left( \mathcal{L},y\right)$, namely a positive constant $\mathfrak{d}$ such that $l\left(\mathcal{L},y\right)\ge \mathfrak{d}$ holds (see \cite{SMA}, Section V.4).
\begin{lemma}[\cite{SMA}, Section V.4]\label{lem2.5m}
	Let $b_1, \dots, b_k$ be an LLL-reduced basis for a lattice $\mathcal{L}$ and $b_1^*, \dots, b_k^*$ be the corresponding Gram-Schmidt orthogonal basis. Let $y\in\mathbb{R}^k$ and $z=B^{-1}y$.
	\begin{enumerate}[\upshape(i)]
		\item If $y\not \in \mathcal{L}$, let $i_0$ be the largest index such that $z_{i_0}\ne 0$ and put $\lambda:=\{z_{i_0}\}$.
		\item If $y\in \mathcal{L}$, put $\lambda:=1$.
	\end{enumerate}
	Now, define
	\[
	c_1 := \max_{1\le j\le k}\left\{\dfrac{\|b_1\|}{\|b_j^*\|}\right\}.
	\]
	Then, $l(\mathcal{L}, y) \ge  \mathfrak{d} := \lambda\|b_1\|c_1^{-1}$.
\end{lemma}

In our application, we are given real numbers $\eta_0,\eta_1,\ldots,\eta_k$ which are linearly independent over $\mathbb{Q}$ and two positive constants $c_3$ and $c_4$ such that 
\begin{align}\label{2.9m}
	|\eta_0+x_1\eta_1+\cdots +x_k \eta_k|\le c_3 \exp(-c_4 H),
\end{align}
where the integers $x_i$ are bounded as $|x_i|\le X_i$ with $X_i$ given upper bounds for $1\le i\le k$. We write $X_0:=\max\limits_{1\le i\le k}\{X_i\}$. The basic idea in such a situation, from \cite{Weg}, is to approximate the linear form \eqref{2.9m} by an approximation lattice. So, we consider the lattice $\mathcal{L}$ generated by the columns of the matrix
$$ \mathcal{A}=\begin{pmatrix}
	1 & 0 &\ldots& 0 & 0 \\
	0 & 1 &\ldots& 0 & 0 \\
	\vdots & \vdots &\vdots& \vdots & \vdots \\
	0 & 0 &\ldots& 1 & 0 \\
	\lfloor E_{\rm lat}\eta_1\rfloor & \lfloor E_{\rm lat}\eta_2\rfloor&\ldots & \lfloor E_{\rm lat}\eta_{k-1}\rfloor& \lfloor E_{\rm lat}\eta_{k} \rfloor
\end{pmatrix} ,$$
where $E_{\rm lat}$ is a large constant usually of the size of about $kX_0^k$. Let us assume that we have an LLL--reduced basis $b_1,\ldots, b_k$ of $\mathcal{L}$ and that we have a lower bound $l\left(\mathcal{L},y\right)\ge \mathfrak{d}$ with $y:=(0,0,\ldots,-\lfloor E_{\rm lat}\eta_0\rfloor)$. Note that $ \mathfrak{d}$ can be computed by using the results of Lemma \ref{lem2.5m}. Then, with these notations the following result  is Lemma VI.1 in \cite{SMA}.
\begin{lemma}[Lemma VI.1 in \cite{SMA}]\label{lem2.6m}
	Let $S:=\displaystyle\sum_{i=1}^{k-1}X_i^2$ and $T:=\dfrac{1+\sum_{i=1}^{k}X_i}{2}$. If $\mathfrak{d}^2\ge T^2+S$, then inequality \eqref{2.9m} implies that we either have $x_1=x_2=\cdots=x_{k-1}=0$ and $x_k=-\dfrac{\lfloor E_{\rm lat}\eta_0 \rfloor}{\lfloor E_{\rm lat}\eta_k \rfloor}$, or
	\[
	H\le \dfrac{1}{c_4}\left(\log(E_{\rm lat}c_3)-\log\left(\sqrt{\mathfrak{d}^2-S}-T\right)\right).
	\]
\end{lemma}

Lastly here, we recall one additional simple fact from calculus. If $t\in \mathbb{R}$ satisfies $|t|<1/2$, then 
\begin{align}\label{eq:calc}
	|\log(1+t)|&<|t-t^2/2+-\dots|
	<|t|+\frac{|t|^2+|t|^3+\dots}2
	<|t|\left(1+\frac{|t|}{2(1-|t|)}\right)<\frac 32 |t|.
\end{align}
In a similar way, we obtain the lower bound $|\log(1+t)|>\frac 12 |t|$ provided that $|t|<1/2$. We shall use these inequalities later in the proof of the main result.

All computational verifications in this work were performed using \textsc{SageMath 10.6}.

\section{Proof of Theorem \ref{thm:main}}\label{sec:proof}

Throughout the proof we write
\begin{equation*}
	A := \max\{a, b\}.
\end{equation*}

\subsection{Small cases: $A \leq 2200$}

We begin by locating all solutions to inequality \eqref{eq:ineq} in the range $a, b \in [0, 2200]$.\footnote{\label{fn:github} All SageMath code used in this paper is available at \url{https://github.com/hbatte/diophantine-ineq-2a3b-sage}.} For any fixed pair  of integers $(a, b)$ with $a\ge 0, b \geq 0$, any integer $x$ satisfying \eqref{eq:ineq} must lie within distance
$3\max\{a,b\}$ of $\sqrt{2^a 3^b}$. The closest such integers $x$ to ${\sqrt{2^a3^b}}$ is one of the following: 
\begin{equation}\label{eq:two-candidates}
	\lfloor\sqrt{2^a 3^b}\rfloor,\;
	\lfloor\sqrt{2^a 3^b}\rfloor + 1.
\end{equation}
We verify directly using \textsc{SageMath} that among all pairs $(a,b)$ with $0 \leq a, b \leq 2200$ and some candidate value $x$ given by \eqref{eq:two-candidates}, the inequality
\begin{align*}	
	1 \leq |x^2 - 2^a \cdot 3^b| < 3\max\{a,b\},
\end{align*}
holds for exactly $42$ pairs $(a,b)$. Later, for each of these pairs, we find all possibilities for the nonnegative integers $x$ participating in $1\le |x^2-2^a3^b|<3\max\{a,b\}$ obtaining
the $57$ triples $(a, b, x)$ listed in
Table \ref{tab:small}. 
\subsection{Large cases: $A > 2200$}\label{sec:large}

From now on we assume $A > 2200$. We first dispose of the case where $a$ and $b$ are both even.

\medskip

\textit{The case $a \equiv b \equiv 0\pmod{2}$.}
Write $a = 2\alpha$ and $b = 2\beta$, so that $2^a 3^b = (2^\alpha 3^\beta)^2$ is a perfect square. Factoring the left-hand side of \eqref{eq:ineq} gives
\begin{equation*}
	|x - 2^\alpha 3^\beta|\cdot|x + 2^\alpha 3^\beta|
	= |x^2 - 2^a 3^b| < 3A.
\end{equation*}
Dividing both sides by $x + 2^\alpha 3^\beta \geq 2^\alpha 3^\beta \geq 2^{A/2}$, we get
\begin{equation*}
	|x - 2^\alpha 3^\beta|
	< \frac{3A}{x + 2^\alpha 3^\beta}
	\leq \frac{3A}{2^\alpha 3^\beta}
	\leq \frac{3A}{2^{A/2}} < 1
	\qquad\text{for all } A > 2200.
\end{equation*}
Since $x$ and $2^\alpha 3^\beta$ are both positive integers, this forces $x = 2^\alpha 3^\beta$, giving $|x^2 - 2^a 3^b| = 0$. But then the condition $1 \leq |x^2 - 2^a 3^b|$ fails. 
\medskip

For the remainder of the proof we therefore assume $A > 2200$ and
\begin{align*}
	(a,b) \equiv (1,0),\; (0,1),\; \text{or}\; (1,1) \pmod{2}.
\end{align*}
In all three cases we write
\begin{equation}\label{eq:d-decomp}
	2^a 3^b = 2^{2\alpha} 3^{2\beta} \cdot d,
	\qquad
	\alpha := \lfloor a/2 \rfloor,\quad
	\beta  := \lfloor b/2 \rfloor,\quad
	d \in \{2, 3, 6\},
\end{equation}
where $d = 2$ when $(a,b)\equiv(1,0)$, $d=3$ when $(a,b)\equiv(0,1)$, and $d = 6$ when $(a,b)\equiv(1,1)\pmod{2}$.

\medskip

We first reduce \eqref{eq:ineq} to a rational approximation problem. Substituting \eqref{eq:d-decomp} into \eqref{eq:ineq} and factoring, we obtain
\begin{equation}\label{eq:factor}
	\lvert x - 2^{\alpha}3^{\beta}\sqrt{d}\,\rvert
	\cdot
	\lvert x + 2^{\alpha}3^{\beta}\sqrt{d}\,\rvert < 3A.
\end{equation}
Let $d_0 := \gcd(x, 2^{\alpha}3^{\beta}) = 2^{\alpha_0}3^{\beta_0}$,
and write $x =: d_0 x_0$, $2^{\alpha}3^{\beta} =: d_0 y_0$, where $\gcd(x_0, y_0) = 1$ and 
$$y_0 = 2^{\alpha-\alpha_0}3^{\beta-\beta_0}.$$
Dividing \eqref{eq:factor} through by $d_0^2 y_0(x_0 + y_0\sqrt{d})$ gives
\begin{align*}
	\left|\frac{x_0}{y_0} - \sqrt{d}\right|
	< \frac{3A}{d_0^2 y_0(x_0 + y_0\sqrt{d})}.
\end{align*}
We bound the right-hand side uniformly in two sub-cases.

\medskip

\textit{Sub-case $x_0/y_0 > \sqrt{d}$.}
Then $x_0 + y_0\sqrt{d} > 2\sqrt{d}\,y_0 \geq 2\sqrt{2}\,y_0$, so
\begin{equation*}
	\frac{3A}{d_0^2 y_0(x_0+y_0\sqrt{d})}
	< \frac{3A}{2\sqrt{2}(d_0 y_0)^2}
	< \frac{1.07A}{(d_0 y_0)^2}.
\end{equation*}

\medskip

\textit{Sub-case $x_0/y_0 < \sqrt{d}$.}
Since $d_0 y_0 = 2^\alpha 3^\beta \geq 2^{A/2}$ and $A > 2200$, then
\begin{equation*}
	\frac{x_0}{y_0}
	> \sqrt{d}\left(1-\frac{3A}{(d_0 y_0)^2}\right)
	> \sqrt{d}\left(1-\frac{3A}{2^A}\right)
	> \sqrt{d}\left(1-\frac{1}{10^{25}}\right),
\end{equation*}
so $x_0 + y_0\sqrt{d} > 2\sqrt{d}\,y_0(1-10^{-25})$, and
\begin{equation*}
	\frac{3A}{d_0^2 y_0(x_0+y_0\sqrt{d})}
	< \frac{3A(1-10^{-25})^{-1}}{2\sqrt{d}(d_0 y_0)^2}
	< \frac{1.07A}{(d_0 y_0)^2},
\end{equation*}
where we used $d \geq 2$.

In both sub-cases, the same bound holds. Setting
\begin{align*}
	K := \frac{1.07A}{d_0^2},
\end{align*}
we obtain in all cases
\begin{equation}\label{eq:worley-input}
	\left|\frac{x_0}{y_0} - \sqrt{d}\right| < \frac{K}{y_0^2}.
\end{equation}
If $K < 1/2$, then $x_0/y_0$ is a convergent to $\sqrt{d}$ by the classical theory of continued fractions, and this case is handled separately later. We therefore assume
\begin{equation}\label{eq:K-half}
	K \geq 1/2, \qquad\text{equivalently}\qquad
	d_0^2 \leq 2.14A, \qquad d_0 \leq \sqrt{2.14A}.
\end{equation}

\medskip

Applying Theorem \ref{thm:Worley} to inequality \eqref{eq:worley-input}, we conclude that either $x_0/y_0 = p_n/q_n$ is a convergent of $\sqrt{d}$, or it has one of the three forms in Theorem \ref{thm:Worley}. In all cases there exist integers $r \geq 1$, $s \geq 1$, $\varepsilon \in \{\pm 1\}$ with $rs < 2K$, and an index $n \geq 1$, such that
\begin{equation}\label{eq:y0-param}
	y_0 = rq_n + \varepsilon sq_{n-1},
\end{equation}
where $p_n/q_n$ denotes the convergents of $\sqrt{d}$.

\medskip

The convergent denominators $q_n$ for each $d$ are given by the following Binet-type formulas.

\medskip

\textbf{Case $d=2$.} Since $\sqrt{2} = [1;\overline{2}]$, all partial quotients are equal to $2$ and Form {\rm(iii)} of Theorem \ref{thm:Worley} never occurs. The denominators $1,2,5,12,29,\ldots$ satisfy $q_{n+1}=2q_n+q_{n-1}$ for all $n\ge 1$, giving 
\begin{align*}
	q_n = \frac{\gamma^{n+1}-\delta^{n+1}}{2\sqrt{2}},
	\qquad (\gamma,\delta) := (1+\sqrt{2},\;1-\sqrt{2}).
\end{align*}
Note that here $\gamma\delta=-1$, so $\delta=-\gamma^{-1}$.  

\textbf{Case $d=3$.}
Here, $\sqrt{3} = [1;\overline{1,2}]$, with all odd-indexed partial quotients equal to $1$, so Form {\rm(iii)} occurs for even $n$. The denominators $1,1,3,4,11,15,41,56,\ldots$ split into two interleaved subsequences $(q_{2m})$ and $(q_{2m+1})$, each satisfying $q_{n+2} = 4q_n - q_{n-2}$ (with characteristic equation $x^2-4x+1=0$) for all $n\ge 2$. We have
\begin{align*}
	q_{2m+i} = \frac{\mathfrak{a}_i\gamma^m - \mathfrak{b}_i\delta^m}{2\sqrt{3}},
	\quad i \in \{0,1\},\qquad
	(\gamma,\delta): = (2+\sqrt{3},\;2-\sqrt{3}),
\end{align*}
where
\begin{align*}
	(\mathfrak{a}_0,\mathfrak{b}_0) := (2+\sqrt{3},\;2-\sqrt{3}),
	\qquad
	(\mathfrak{a}_1,\mathfrak{b}_1) := (1+\sqrt{3},\;1-\sqrt{3}).
\end{align*}
Note that $\gamma\delta = 1$ here, so $\delta = \gamma^{-1}$.

\medskip

\textbf{Case $d=6$.} Here, $\sqrt{6} = [2;\overline{2,4}]$, with all partial quotients greater than or equal to $2$, so
Form~{\rm(iii)} never occurs. The denominators $1,2,9,20,89,198,\ldots$ similarly split into two interleaved binary recurrent subsequences (the subsequence of even indices and the subsequence of odd indices), each satisfying the recurrence $q_{n+2} = 10q_n - q_{n-2}$ (with characteristic equation $x^2-10x+1=0$) for all $n\ge 2$. We have
\begin{align*}
	q_{2m+i} = \frac{\mathfrak{a}_i\gamma^m - \mathfrak{b}_i\delta^m}{2\sqrt{6}},
	\quad i \in \{0,1\},\qquad
	(\gamma,\delta) := (5+2\sqrt{6},\;5-2\sqrt{6}),
\end{align*}
where
\begin{align*}
	(\mathfrak{a}_0,\mathfrak{b}_0): = (2+\sqrt{6},\;2-\sqrt{6}),
	\qquad
	(\mathfrak{a}_1,\mathfrak{b}_1) := (5+2\sqrt{6},\;5-2\sqrt{6}).
\end{align*}
Note again that $\gamma\delta = 1$, so $\delta=\gamma^{-1}$.

\medskip 

Next, we find a unified equation for each case on $d\in\{2,3,6\}$. Substituting the appropriate Binet formula into \eqref{eq:y0-param},
collecting powers of $\gamma$ and $\delta$, and multiplying through by
$2\sqrt{d}$, we obtain in every case a relation of the form
\begin{equation}\label{eq:CD}
	2\sqrt{d}\,y_0 = C\gamma^n + D\delta^n,
\end{equation}
where $C$ and $D$ are algebraic integers in $\mathcal{O}_{\mathbb{K}}$
(with $\mathbb{K} = \mathbb{Q}(\sqrt{d})$) that are conjugates under the
field automorphism $\sqrt{d}\mapsto -\sqrt{d}$.
Table~\ref{tab:CD} records the six parametric families.

\begin{table}[H]
	\centering
	
	\label{tab:CD}
	\renewcommand{\arraystretch}{1.4}
	\begin{tabular}{cllll}
		\toprule
		$d$ & $C$ & $D$ & $n$ & Condition \\
		\midrule
		$2$ & $r\gamma+\varepsilon s$
		& $-(r\delta+\varepsilon s)$
		& $n=m$
		& \\[4pt]
		$3$ & $r\mathfrak{a}_1+\varepsilon s\mathfrak{a}_0$
		& $r\mathfrak{b}_1+\varepsilon s\mathfrak{b}_0$
		& $m=2n+1$ (odd $m$)
		& \\[4pt]
		$3$ & $r\mathfrak{a}_0\gamma+\varepsilon s\mathfrak{a}_1$
		& $-(r\mathfrak{b}_0\delta+\varepsilon s\mathfrak{b}_1)$
		& $m=2n+2$ (even $m$)
		& \\[4pt]
		$3$ & $r\mathfrak{a}_1+s\gamma^{-1}$
		& $-(r\mathfrak{b}_1+s\delta^{-1})$
		& $m+1=2n+1$
		& Form~{\rm(iii)}: $a_{m+1}=1$ \\[4pt]
		$6$ & $r\mathfrak{a}_1+\varepsilon s\mathfrak{a}_0$
		& $-(r\mathfrak{b}_1+\varepsilon s\mathfrak{b}_0)$
		& $m=2n+1$ (odd $m$)
		& \\[4pt]
		$6$ & $r\mathfrak{a}_0\gamma+\varepsilon s\mathfrak{a}_1$
		& $-(r\mathfrak{b}_0\delta+\varepsilon s\mathfrak{b}_1)$
		& $m=2n+2$ (even $m$)
		& \\
		\bottomrule
	\end{tabular}
	\caption{The six families for $(C,D,n)$ in equation~\eqref{eq:CD}.}
\end{table}

\noindent
Note: In all cases $r\geq 1$, $s\geq 1$, $\gcd(r,s)=1$,
$\varepsilon\in\{\pm 1\}$, $rs<2K$. In Row 4 (Form {\rm(iii)})
we have $rs < K$, $\gamma^{-1}=\delta$ and $\delta^{-1}=\gamma$ since $\gamma\delta=1$ for $d=3$.

\medskip

We observe the following global bounds on the data.
\begin{equation}\label{eq:gamma-max}
	\max_{d\in\{2,3,6\}}|\gamma| < 5+2\sqrt{6},\quad
	\max_{\substack{d\in \{2,3,6\},\,i\in\{0,1\}}}|\mathfrak{a}_i| < 5+2\sqrt{6},\quad
	\max_{\substack{d\in \{2,3,6\},\,i\in\{0,1\}}}\{|\delta|,|\mathfrak{b}_i|\} < 1.
\end{equation}
Using \eqref{eq:gamma-max} and $rs<2K$ with $r,s\geq 1$
(so $r\leq 2K/s\leq 2K$ and $s\leq 2K$), inspection of Table \ref{tab:CD} in each of the six rows gives
\begin{equation}\label{eq:C-upper}
	|C| \leq (5+2\sqrt{6})^2 r+(5+2\sqrt{6})s
	< 10\left(10r+\frac{2K}{r}\right)
	\leq 10(20K+1) \leq 220K < 240A,
\end{equation}
and
\begin{equation}\label{eq:D-upper}
	|D| < (5+2\sqrt{6})(r+s)
	< 10\left(r+\frac{2K}{r}\right)
	\leq 10(2K+1) \leq 40K < 43A,
\end{equation}
where we have also used $0.5 \le K \le 1.07A$. Since $CD$ is the norm $N_{\mathbb{K}/\mathbb{Q}}(C) \in \mathbb{Z}$ and
$C \neq 0$, we have $|CD| \geq 1$, giving the lower bounds
\begin{align*}
	|C| \geq \frac{1}{|D|} > \frac{1}{43A}, \qquad
	|D| \geq \frac{1}{|C|} > \frac{1}{240A}.
\end{align*}

\medskip

Next, we bound $n$. From \eqref{eq:CD} and equations \eqref{eq:C-upper}--\eqref{eq:D-upper}, we obtain
\begin{equation*}
	2\sqrt{d}\,y_0 \leq |C|\gamma^n + |D||\delta|^n
	< 240A\,\gamma^n + 43A,
\end{equation*}
and rearranging terms gives
\begin{equation}\label{eq:gamma-lower}
	\gamma^n > \frac{2\sqrt{d}\,y_0 - 43A}{240A}
	> \frac{0.99\times 2\sqrt{d}}{240A}\,y_0.
\end{equation}
The last inequality above comes from the fact that 
\begin{align*}
	2{\sqrt{d}}y_0-43A>0.99\times 2{\sqrt{d}} y_0,
\end{align*}
which indeed holds because $ 0.01 (2{\sqrt{d}} y_0)-43A>0$. In fact, this is implied by 
\begin{eqnarray*}
	\label{eq:14}
	0.01 (2{\sqrt{d}} y_0) -43A & \ge & \frac{0.01 \times (2{\sqrt{2}}) 2^{(A-1)/2}}{d_0}-43A\nonumber\\
	& > & \frac{0.01\times ({2{\sqrt{2}}}) 2^{(A-1)/2}}{{\sqrt{2.14A}}}-43A>0,
\end{eqnarray*}
where the last inequality holds for all $A>2200$. From the other direction of \eqref{eq:CD},
\begin{align*}
	2{\sqrt{d}} y_0=|C\gamma^n-D\delta^n|>|C|\gamma^n-|D|>\frac{\gamma^n}{43A}-240A.
\end{align*}
This gives
\begin{equation}\label{eq:gamma-upper}
	\frac{\gamma^n}{43A} < 2\sqrt{d}\,y_0 + 240A
	< 1.01\times 2\sqrt{d}\,y_0,
\end{equation}
where the last inequality above follows by the same argument used to explain \eqref{eq:gamma-lower}. Combining \eqref{eq:gamma-lower} and \eqref{eq:gamma-upper}, we get:
\begin{align*}
	\frac{y_0}{86A} < \gamma^n < 213\,y_0\,A.
\end{align*}
Taking logarithms throughout and simplifying, we get
\begin{align*}
	\log y_0 - \log(86A) < n\log\gamma < \log y_0 + \log(213A).
\end{align*}
Using $y_0 \leq 2^\alpha 3^\beta$ and $A = \max\{a,b\} \geq 2\max\{\alpha,\beta\}$,
the upper bound becomes
\begin{align}\label{eq:n-upper}
	n < \frac{\tfrac{A}{2}\log 6 + \log(213A)}{\log\gamma}.
\end{align}
The right-hand side of \eqref{eq:n-upper} is maximised when $\log\gamma$ is smallest; i.e., for $d=2$ where $\gamma = 1+\sqrt{2}$. Therefore, uniformly for all $d \in \{2,3,6\}$, we have
\begin{align*}
	n < \frac{\tfrac{A}{2}\log 6+\log(213A)}{\log(1+\sqrt{2})}
	< 1.02A + 6.09 + 1.14\log A.
\end{align*}

We record these results for future reference.

\begin{lemma}\label{lem:n-bound}
	Assume $A > 2200$, let $K = 1.07A/d_0^2\ge 1/2$ with
	$d_0 = 2^{\alpha_0}3^{\beta_0}$, and assume $K \geq 1/2$.
	Set $y_0 = 2^{\alpha-\alpha_0}3^{\beta-\beta_0}$.
	Then $d_0 \leq \sqrt{2.14A}$ and, for all $d \in \{2,3,6\}$,
	\begin{equation*}
		\log y_0 - \log(86A) < n\log\gamma < \log y_0+\log(213A).
	\end{equation*}
	In particular, $n < 1.02A + 6.09 + 1.14\log A$.
\end{lemma}

\medskip

Next, we take $p$-adic valuations of \eqref{eq:CD} for $p\in \{2,3\}$. We rewrite \eqref{eq:CD} by factoring out $C\delta^n$ as
\begin{align*}
	2\sqrt{d}\,y_0 = C\delta^n\!\left(\left(\frac{\gamma}{\delta}\right)^{\!n}
	+ \frac{D}{C}\right).
\end{align*}
Now $\gamma/\delta$ equals $-\gamma^2$ when $d=2$ (since $\gamma\delta=-1$) and $\gamma^2$ when $d\in\{3,6\}$ (since $\gamma\delta=1$). Writing $\varepsilon_1 \in \{\pm 1\}$ for this sign, we have
\begin{equation}\label{eq:factor-sign}
	2\sqrt{d}\,y_0 = C\delta^n\!\left(\gamma^{2n}+\varepsilon_1\frac{D}{C}\right),
	\quad
	\varepsilon_1 = \begin{cases}-1 & d=2,\;n\text{ odd},\\ +1 & \text{otherwise.}\end{cases}
\end{equation}

Let $\pi_1$ and $\pi_2$ denote the prime ideals of $\mathcal{O}_{\mathbb{K}}$ above the rational primes $p_1 = 2$ and $p_2 = 3$ respectively. The factorizations of $2$ and $3$ in $\mathcal{O}_{\mathbb{K}}$ are
\begin{equation}\label{eq:factorizations}
	\begin{alignedat}{3}
		d=2{:}\quad
		& 2 = (\sqrt{2})^2 =: \pi_1^2, &\qquad& 3 \text{ is inert} =: \pi_2;\\
		d=3{:}\quad
		& 2 = (1+\sqrt{3})^2(2-\sqrt{3}) =: \pi_1^2\delta, &\qquad&
		3 = (\sqrt{3})^2 =: \pi_2^2;\\
		d=6{:}\quad
		& 2 = (2+\sqrt{6})^2(5-2\sqrt{6}) =: \pi_1^2\delta, &\qquad&
		3 = (3+\sqrt{6})^2(5-2\sqrt{6}) =: \pi_2^2\delta.
	\end{alignedat}
\end{equation}
The ramification indices are
\begin{align*}
	e_1 = 2 \text{ (for all } d\text{)};\qquad
	e_2 = \begin{cases}1 & d=2,\\2 & d\in\{3,6\}.\end{cases}
\end{align*}
Formula \eqref{eq:nup} shows that the formulas
$$
\nu_2(\eta)=\frac{1}{2} \nu_{\pi_1}(\eta)\qquad {\text{\rm and}}\qquad \nu_3(\eta)=\frac{1}{2} \nu_{\pi_2}(\eta)
$$
hold in all cases except when $d=2,~p=3$ for which $\nu_3(\eta)=\nu_{\pi_2}(\eta)$. Since $\delta$ is a unit in $\mathcal{O}_{\mathbb{K}}$ (as
$\gamma\delta \in \{\pm 1\}$), taking the $p$-adic valuation of both sides of \eqref{eq:factor-sign} gives
\begin{equation}\label{eq:val-eq}
	\nu_{p}(2\sqrt{d}\,y_0)
	= \nu_{p}(C) + \nu_{p}\!\left(\gamma^{2n}+\varepsilon_1\frac{D}{C}\right),
	\qquad i = 1, 2.
\end{equation}

We begin with the left-hand side of \eqref{eq:val-eq}. From $y_0 = 2^{\alpha-\alpha_0}3^{\beta-\beta_0}$ and \eqref{eq:factorizations},
\begin{equation}\label{eq:lhs-val}
	\nu_{2} (2\sqrt{d}\,y_0) \geq \alpha-\alpha_0+1 \qquad {\text{\rm and}}\qquad \nu_{3}(2\sqrt{d}\,y_0) \geq \beta-\beta_0.
\end{equation}
For the right-hand side of \eqref{eq:val-eq}, we bound $\nu_{p}(C)$ from above. Since $N_{\mathbb{K}/\mathbb{Q}}(C) = CD \in \mathbb{Z}$ and
$$|CD| < 240A \times 43A = 10320A^2,$$
then
\begin{align}\label{eq:norm-bound}
	\nu_{p}(C)
	\leq \frac{\log|N(C)|}{\log p_i}
	< \frac{\log(10320A^2)}{\log 2}
	= \frac{2\log A}{\log 2} + \frac{\log 10320}{\log 2}
	< 2.9\log A + 13.4.
\end{align}
From \eqref{eq:val-eq}, \eqref{eq:lhs-val} and \eqref{eq:norm-bound},
\begin{align}
	\nu_{2}\!\left(\gamma^{2n}+\varepsilon_1\frac{D}{C}\right)
	&\geq \alpha-\alpha_0 - 2.9\log A - 13.4\label{eq:lower1}\\
	\nu_{3}\!\left(\gamma^{2n}+\varepsilon_1\frac{D}{C}\right)
	&\geq \beta-\beta_0 - 2.9\log A - 13.4. \label{eq:lower2}
\end{align}

\medskip

We now apply Theorem \ref{thm:Bugt} with
\begin{equation*}
	\eta_1 := \gamma^2,\qquad
	\eta_2 := \varepsilon_1\frac{D}{C},\qquad
	b_1 := 2n,\quad b_2 := 1,
\end{equation*}
so that $\nu_{p}(\eta_1^{b_1}\eta_2^{b_2}-1) =
\nu_{p}(\gamma^{2n}+\varepsilon_1 D/C)$ whenever $\nu_{p}(\eta_2)=0$ (which we may assume; otherwise the expression has valuation zero and \eqref{eq:lower1}--\eqref{eq:lower2} are vacuous). We assume further that $\eta_{1}$ and $\eta_2$ are multiplicatively independent and take
$D_{\mathbb K} := [\mathbb{K}:\mathbb{Q}] = 2$. We verify the following hypotheses.
\begin{itemize}
	\item $\nu_{p}(\eta_1)=0$: $\gamma^2$ is a unit in
	$\mathcal{O}_{\mathbb{K}}$ since $\gamma$ is the fundamental unit.
	\item $\nu_{p}(\eta_2)=0$: assumed above.
	\item The parameter $g$ (smallest positive integer with
	$\nu_{p}(\eta_1^{g}-1)>0$ for $p=2,3$). A direct calculation gives
	\begin{equation}\label{eq:g-vals}
		g_i = \begin{cases}
			1 & p_i=2 \text{ for all } d\in\{2,3,6\},\\
			1 & p_i=3,\; d\in\{3,6\},\\
			4 & p_i=3,\; d=2.
		\end{cases}
	\end{equation}
For example, when $d=2$ and $p=2$, $\eta_1-1 = 2+2\sqrt{2} = 2(1+\sqrt{2})$ has $\nu_2(\eta_1-1)>0$, so $g=1$. When $d=2$ and $p=3$, one checks $\eta_1^4 \equiv 1\pmod{3}$ while $\eta_1^k\not\equiv 1\pmod{3}$ for $k=1,2,3$, giving $g=4$.
\end{itemize}

\medskip

The minimal polynomials of $\eta_1=\gamma^2$ are $X^2-6X+1$, $X^2-14X+1$, $X^2-98X+1$ for $d=2,3,6$ respectively, each of the form $X^2-((\gamma^2+\delta^2))X+1$. Thus $h(\eta_1) = \log\gamma$ in every case. Additionally, since $\log\gamma \geq \log(1+\sqrt{2}) > (\log p)/2$ for $p\in\{2,3\}$, we have $h'(\eta_1) = \log\gamma$.

On the other hand, $|D/C|$ is a root of the polynomial 
\begin{align*}
	|CD|\left(X-|C/D|)(X-|D/C|\right)
	=  |CD| X^2-(C^2+D^2)X+|CD|\in {\mathbb Z}[X]
\end{align*}
of leading coefficient $|CD|$, one root $>1$ in absolute value and one root $<1$ in absolute value, so, assuming say $|C|\ge |D|$, we have
$$h(D/C) \leq
\frac{1}{2}(\log|CD|+\log|C/D|) = \log|C| \leq \log(240A),$$ (assuming
$|C|\leq|D|$ gives an even better inequality, namely $h(D/C)\le \log |D|\le \log(43A)$). Thus,
\begin{align*}
	h(\eta_2) \leq \log(240A) < \log A + 5.5.
\end{align*}
We set $h'(\eta_2) := \log A + 5.5$.

\medskip

The parameter $E'$ can be computed as 
\begin{align*}
	E' = \frac{b_1}{h'(\eta_2)}+\frac{b_2}{h'(\eta_1)}
	= \frac{2n}{\log A+5.5}+\frac{1}{\log\gamma}.
\end{align*}
Using $n < 1.02A + 6.09 + 1.14\log A$ from Lemma \ref{lem:n-bound} and absorbing the $1/\log\gamma$ term into the numerator, we obtain
\begin{align*}
	E' < \frac{2.04A + 12.18 + 2.28\log A}{\log A+5.5}+\frac{1}{\log\gamma}<\frac{2.14A + 12.18 + 2.28\log A}{\log A+5.5}.
\end{align*}
For $A > 10^6$, the maximum in the definition of $E$ from Theorem \ref{thm:Bugt} is achieved by its logarithmic term
\begin{align*}
	E = \log\!\left(\frac{1.02A+2.3\log A+7.3}{\log A+5.5}\right)+\log\log p+0.4.
\end{align*}
Using the numbers $p$, $g$ from \eqref{eq:g-vals} and $D_{\mathbb K}=2$, we obtain the constant
\begin{align*}
	\mathcal{C}_p:=\frac{24\,p\,g\,D_{\mathbb K}^4}{(p-1)(\log p)^4}
	< \begin{cases}
		3328 & p=2 \text{ (all } d\text{)},\\
		1582 & p=3,\; d=2,\\
		396 & p=3,\; d\in\{3,6\}.
	\end{cases}
\end{align*}
Thus, Theorem \ref{thm:Bugt} gives, for $i=1,2$,
\begin{align}\label{eq:BL-upper}
	\nu_{p}\!\left(\gamma^{2n}+\varepsilon_1\frac{D}{C}\right)
	\leq \mathcal{C}_p\cdot\log\gamma\cdot(\log A+5.5)\cdot E^2.
\end{align}
Combining the lower bound in \eqref{eq:lower1} with worse upper bound for $p=2$, $e=2$, $\mathcal{C}_2<3328$, we get
\begin{align}\label{eq:alpha-bound}
	\alpha-\alpha_0<3328\log\gamma(\log A+5.5)E^2+2.9\log A+13.4.
\end{align}
On the other hand, combining \eqref{eq:lower2} with \eqref{eq:BL-upper} for $p=3$,  $\mathcal{C}_3<1582$, gives
\begin{align*}
	\beta-\beta_0\le 1582\log\gamma(\log A+5.5)E^2+5.8\log A+26.7.
\end{align*}
Since $3328 > 1582$, the bound from \eqref{eq:alpha-bound} dominates. Thus,
\begin{align*}
	\max\{\alpha-\alpha_0,\,\beta-\beta_0\} <
	3328\log\gamma(\log A+5.5)E^2+2.9\log A+13.4.
\end{align*}
We use the chain
\begin{align}\label{eq:chain}
	\frac{A-1}{2}
	&\leq \max\{\alpha,\beta\}
	\leq \max\{\alpha-\alpha_0,\beta-\beta_0\}+\max\{\alpha_0,\beta_0\}\nonumber\\
	&\leq \max\{\alpha-\alpha_0,\beta-\beta_0\}+\frac{\log d_0}{\log 2},
\end{align}
and $\log d_0 \leq \frac{1}{2}\log(2.14A)$ from \eqref{eq:K-half}, giving
\begin{equation}\label{eq:A-ineq}
	\frac{A-1}{2}
	< 3328(\log\gamma)(\log A+5.5)E^2+2.9\log A+13.4	+ \frac{\log(2.14A)}{2\log 2}.
\end{equation}
The right-hand side of \eqref{eq:A-ineq} depends on $d$ only through $\log\gamma$. Larger $\gamma$ gives a weaker bound on $A$. Inserting $\log(1+\sqrt{2})$, $\log(2+\sqrt{3})$,
$\log(5+2\sqrt{6})$ and performing a direct numerical evaluation in \textsc{SageMath} shows that \eqref{eq:A-ineq} fails for $A$ exceeding the values below. (If $A\leq 1\times 10^6$ the bounds hold trivially.)

\medskip

\begin{lemma}\label{lem:A-bound}
	We have
	\begin{align*}
		A := \max\{a,b\} <
		\begin{cases}
			 2.24\times 10^7 & d = 2, \\
			3.62\times 10^7 & d = 3, \\
			7.14\times 10^7 & d = 6.
		\end{cases}
	\end{align*}
	In particular, $A < 7.14\times 10^7$ in all cases.
\end{lemma}
\begin{proof}
	If $A \leq 10^6$ the stated bounds hold trivially. For $A > 10^6$,
	inserting 
	$$\log\gamma \in \{\log(1+\sqrt{2}), \log(2+\sqrt{3}),
	\log(5+2\sqrt{6})\}$$ into inequality~\eqref{eq:A-ineq} and evaluating
	numerically in \textsc{SageMath} yields the bounds above.
\end{proof}

Combined with Lemma \ref{lem:n-bound}, this gives
\begin{align*}
	n <
	\begin{cases}
		2.29\times 10^7 & d=2,\\
		3.70\times 10^7 & d=3,\\
		7.29\times 10^7 & d=6.
	\end{cases}
\end{align*}

\subsection{Reduction via the LLL algorithm}

We now apply the LLL algorithm to reduce the bounds in
Lemma \ref{lem:A-bound} to a range which can allow computational verification.

\medskip

We return to equation \eqref{eq:CD}, which we write as
\begin{align*}
	2\sqrt{d}\,y_0 = |C|\gamma^n \pm |D||\delta|^n.
\end{align*}
Rearranging and dividing both sides by $|C|\gamma^n$, we obtain
\begin{equation}\label{eq:LLL}
	\left|\frac{2^{\alpha-\alpha_0}3^{\beta-\beta_0}}{|C|\gamma^n} - 1\right|
	= \frac{|D/C|}{\gamma^{2n}}.
\end{equation}
We bound the right-hand side of \eqref{eq:LLL}. From the
bounds \eqref{eq:C-upper}--\eqref{eq:D-upper}, we have
\begin{equation*}
	|D/C| \leq |D| \cdot \frac{1}{|C|} \leq 43A \times 43A =  1849A^2,
\end{equation*}
and the lower bound  for $\gamma^n$ from Lemma \ref{lem:n-bound} gives
\begin{equation*}
	\gamma^n > \frac{y_0}{86A} = \frac{2^\alpha 3^\beta}{d_0(86A)}
	\geq \frac{2^{A/2}}{86A\sqrt{2.14A}}.
\end{equation*}
Substituting this into the right-hand side of \eqref{eq:LLL} gives
\begin{equation*}
	\frac{|D/C|}{\gamma^{2n}}
	< \frac{1849A^2 \cdot (86A)^2 \cdot 2.14A}{2^A}
	< \frac{2.93\times 10^7\, A^5}{2^A}.
\end{equation*}
Since $2^{0.05A} > 2.93\times 10^7 A^5$ for all $A > 2200$, the right-hand side above is bounded by $2^{-0.95A}$. Therefore,
\begin{equation}\label{eq:LLL-abs}
	\left|\frac{2^{\alpha-\alpha_0}3^{\beta-\beta_0}}{|C|\gamma^n} - 1\right|
	< \frac{1}{2^{0.95A}}.
\end{equation}
Taking logarithms and using \eqref{eq:calc} (which holds by \eqref{eq:LLL-abs} for $A > 2200$), we obtain
\begin{align}\label{eq:LLL1}
	\left|(\alpha-\alpha_0)\log 2 + (\beta-\beta_0)\log 3
	- n\log\gamma - \log|C|\right|
	< \frac{2}{2^{0.95A}}.
\end{align}
For each fixed value of $|C|$ (arising from a specific choice of $d$, $r$, $s$, $\varepsilon$ from Table \ref{tab:CD}), we use the LLL-algorithm to obtain a lower bound for the smallest nonzero value of the above linear form, constrained by integer coefficients with absolute values not exceeding $n < 7.29 \cdot 10^{7}$. In particular, we consider the lattice  
\begin{equation*}
	\mathcal{A} = \begin{pmatrix}
		1 & 0 & 0 \\
		0 & 1 & 0 \\
		\lfloor E_{\rm lat}\log 2\rfloor &
		\lfloor E_{\rm lat}\log 3\rfloor &
		\lfloor E_{\rm lat}\log\gamma\rfloor
	\end{pmatrix},
\end{equation*}
where $E_{\rm lat} := 1.2\times 10^{24}$ is chosen so that the lattice entries are integers of size comparable to $3n^3$. We set $y := (0,0,-\lfloor E_{\rm lat}\cdot\log|C|\rfloor)^T$. Applying Lemma \ref{lem2.5m}, we obtain using \textsc{SageMath}\footnote{See \texttt{lll\_reduction.ipynb} in the repository cited on page \pageref{fn:github}.} that among all possible values of $|C|$ arising from a specific choice of $(d,r,s,\varepsilon)$,
\[
 l\left(\mathcal{L},y\right)\ge \mathfrak{d} := 2.35\cdot 10^{8}. 
\]
This computation, which scans over all admissible coprime pairs $(r,s)$ for each $d \in \{2,3,6\}$ subject to the conditions described below, took approximately 6 days of continuous running time on a standard desktop machine.
\medskip

The value $|C|$ used in the LLL-algorithm above depends on $d$, $r$, $s$, $\varepsilon$, and $n$ via Table \ref{tab:CD}. From~\eqref{eq:C-upper} and $K \leq 1.07A$, we have $|C| < 240A < 10^{10}$. For each $d \in \{2,3,6\}$ and each pair $(r,s)$ with $r,s \geq 1$, $\gcd(r,s) = 1$, and $rs < 2K \leq 2.14A < 2\times 10^8$, the number of coprime pairs is at most
of order $A (\log A+1) /(\pi^2/6) < 3\times 10^{9}$. However, inequality \eqref{eq:LLL1} suggest that our lattice has dimension $4$. Clearly, $2,3,\gamma$ are multiplicatively independent, but $2,3,\gamma,|C|$ might be multiplicatively dependent. Note that if this is so, then necessarily $|N(C)|=2^u3^v$, where $N$ is the norm in the corresponding quadratic field ${\mathbb K}={\mathbb Q}({\sqrt{d}})$. If this is so, then thanks to the fact that $p=2$ and $p=3$ are always squares in ${\mathbb K}$  except for the case $(p,d)=(3,2)$ when $p=3$ is inert in ${\mathbb K}$, we get the following scenarios:
$$
C=\pm {\sqrt{2}}^u 3^{v/2} \gamma^t,\quad C=\pm (1+{\sqrt{3}})^u {\sqrt{3}}^v \gamma^t,\quad C=\pm (2+{\sqrt{6}})^u (3+{\sqrt{6}})^v\gamma^t,
$$
according to whether $d=2,3,6$, respectively. However, in these cases we also have 
$$
C/D=\pm \gamma^{t},\quad C/D=\pm \gamma^{2t+u},\quad C/D=\pm \gamma^{2t+u+v},
$$
so $\eta_1$ and $\eta_2$ are multiplicatively independent, a case which will be treated in the next section. 

So, when building the lattices to apply {\text{\rm LLL} to \eqref{eq:LLL}, we restrict our attention to values of $C$ such that $|N(C)|$ is divisible by some prime $q\ge 5$.
	
	\medskip 
	
Using Lemma \ref{lem2.6m}, we conclude that $S =1.6 \cdot 10^{16}$ and $T = 1.1 \cdot 10^{8}$. Since $\mathfrak{d}^2 \geq T^2 + S$, we select $c_3 := 2$, $c_4 := \log 2$ and establish the bound $0.95A \leq 54$. Thus, $A<57$, contradicting the working assumption that $A>2200$.

\subsection{Special cases}
Along the way we made some assumptions. For example, assuming still that $rs\ne 0$, we assumed that $\eta_1$ and $\eta_2$ are multiplicatively independent in order to apply Theorem \ref{thm:Bugt}. But if they are not independent, we get even better bounds on $\nu_{p}(\gamma^n+\varepsilon_1 D/C)$ for $p=2,3$. We handle these cases here.

\subsubsection*{Case 1: $\eta_1$ and $\eta_2$ multiplicatively dependent}

Since $\gamma$ is the fundamental unit of $\mathbb{Q}(\sqrt{d})$, the only units are $\pm\gamma^t$ for $t \in \mathbb{Z}$. If $\eta_1 = \gamma^2$ and $\eta_2 = \varepsilon_1 D/C$ are multiplicatively dependent, then
$\eta_2 = \pm\gamma^t$ for some integer $t$, i.e.,
\begin{align*}
	|t|\log\gamma = |\log|D/C||.
\end{align*}
Using $|C| \leq 240A$ and $|D| \leq 43A$, we have 
$$
\log (|D|/|C|)\le \log |D|+\log(1/|C|)\le \log(43A)+\log(43A)\le 2\log(43A),
$$
and
$$
-\log|D|/|C|=\log|C|/|D|\le \log |C|+\log(1/|D|)\le \log(240A)+\log(240A)=2\log (240A).
$$
So, in all cases 
\begin{align*}
	|t| \leq \frac{2\log(240A)}{\log\gamma} < 2.3\log A + 12.5.
\end{align*}
In this case,
\begin{equation*}
	\gamma^{2n} + \varepsilon_1\frac{D}{C}
	= \gamma^{2n} \pm \gamma^t
	= \gamma^t\!\left(\gamma^{2n-t} \pm 1\right),
\end{equation*}
so $\nu_{\pi_i}(\gamma^{2n} + \varepsilon_1 D/C) = \nu_{\pi_i}(\gamma^{2n-t}\pm 1)$, since $\nu_{\pi_i}(\gamma^t) = 0$ for $i=1,2$. As for the other valuation, we have 
$$
\nu_{\pi_i}(\gamma^{2n-t}\pm 1)\le \nu_{\pi_i}((\gamma^{12})^{2n-t}-1)\qquad {\text{\rm for}}\qquad i=1,2.
$$
The exponent $12$ is chosen as the smallest positive integer $k$ such that $\gamma^k$ satisfies the hypothesis of Lemma \ref{lem:BiLu} simultaneously for all $d\in\{2,3,6\}$ and both primes $p\in\{2,3\}$, as verified in \eqref{eq:zeta-vals}. So, we apply Lemma \ref{lem:BiLu} with $\zeta = \gamma^{12}$. One verifies the hypothesis $\nu_{\pi_i}(\zeta-1) > e_i/(p_i-1)$ for each $d \in \{2,3,6\}$ and $i=1,2$ by direct calculation:
\begin{equation}\label{eq:zeta-vals}
	\begin{array}{lll}
		d=2: & \nu_{\pi_1}(\zeta-1) = 7 > 2, & \nu_{\pi_2}(\zeta-1) = 2 > 1/2; \\[4pt]
		d=3: & \nu_{\pi_1}(\zeta-1) = 8 > 2, & \nu_{\pi_2}(\zeta-1) = 3 > 1; \\[4pt]
		d=6: & \nu_{\pi_1}(\zeta-1) = 9 > 2, & \nu_{\pi_2}(\zeta-1) = 5 > 1.
	\end{array}
\end{equation}
The hypothesis is satisfied in all cases as shown in \eqref{eq:zeta-vals}, so Lemma \ref{lem:BiLu} applies with $m = 2n-t$. We get
\begin{equation*}
	\nu_{\pi_i}(\gamma^{2n-t}\pm 1)
	\leq \nu_{\pi_i}(\gamma-1) + \nu_{\pi_i}(2n-t)\le 9+\frac{2\log(|2n-t|)}{\log p_i}
	\leq 9 + \frac{2\log(2n+|t|)}{\log 2}
\end{equation*}
uniformly for all $i=1,2$ and $d\in\{2,3,6\}$
(using $\nu_{\pi_i}(\gamma-1) \leq 9$ in all cases). From \eqref{eq:val-eq} and \eqref{eq:norm-bound}, we have
\begin{equation*}
	\max\{\alpha-\alpha_0,\,\beta-\beta_0\}
	\leq \nu_{\pi_i}(C) + \nu_{\pi_i}(\gamma^{2n-t}\pm 1)
	\leq 5.8\log A + 26.7 + 9 + \frac{2\log(2n+|t|)}{\log 2},\quad i=1,2.
\end{equation*}
Using the chain \eqref{eq:chain}, together with
$n < 1.02A + 6.09 + 1.14\log A$ and $|t| < 2.3\log A + 12.5$, we obtain 
\begin{equation*}
	\frac{A-1}{2} \leq 5.8\log A + 35.7 + \frac{\log\sqrt{2.14A}}{\log 2}
	+ \frac{2\log(2.04A + 4.6\log A + 25)}{\log 2}.
\end{equation*}
This gives $A < 180$, contradicting $A > 2200$.

\subsubsection*{Case 2: $rs = 0$}

If $r = 0$ or $s = 0$ then $y_0 = q_n$ (or $q_{n-1}$) and equation \eqref{eq:CD} becomes
\begin{equation*}
	2\sqrt{d}\,y_0 = C\!\left(\gamma^{2m} + \frac{D}{C}\right),\qquad m\in \{n,n-1\},
\end{equation*}
where
\begin{equation*}
	\frac{D}{C} \in \left\{1,\;
	\frac{1+\sqrt{3}}{1-\sqrt{3}},\;
	\frac{2+\sqrt{6}}{2-\sqrt{6}}\right\}
	= \left\{1,\; -\gamma_{d=3},\; -\gamma_{d=6}\right\}.
\end{equation*}
In each case $D/C$ is a power of the fundamental unit $\gamma$, so $\eta_1 = \gamma^2$ and $\eta_2 = D/C$ are multiplicatively dependent. This is exactly the situation treated in Case 1 above, which gives $A < 180$, again contradicting $A > 2200$.

\medskip

Since both Case 1 and Case 2 lead to $A < 180< 2200$, all solutions in these cases are already accounted for by the small-cases computation of Table \ref{tab:small}. This completes the proof of Theorem \ref{thm:main}. \qed

\section*{Acknowledgments} 
The second author thanks the Mathematics division of Stellenbosch University for funding his PhD studies. The fourth author was partially supported by the 2024 ERC Synergy Grant ``DynAMiCs".

\section*{Addresses}

$ ^{1} $ Bilecik \c{S}eyh Edebali University,
Vocational School, 11200 Bilecik, Turkey.

Email: \url{banu.irez@bilecik.edu.tr}
\\
$ ^{2} $ Mathematics Division, Stellenbosch University, Stellenbosch, South Africa.

Email: \url{hbatte91@gmail.com}

Email: \url{fluca@sun.ac.za}
\\
$ ^{3} $ Bilecik \c{S}eyh Edebali University,
Department of Mathematics, Faculty of Science,
11200 Bilecik, Turkey.

Email: \url{ilker.inam@bilecik.edu.tr}
\\
$ ^{4} $ Max Planck Institute for Software Systems, Saarbr\"ucken, Germany.
\\
$ ^{5} $ Van Y\"{u}z\"{u}nc\"{u} Y{\i}l University,
Muradiye Vocational School, 65080 Tu\c{s}ba, Van, Turkey.

Email: \url{zeynepdemirkolozkaya@yyu.edu.tr}

\section*{Competing interests}
On behalf of all authors, the corresponding author states that there is no Conflict of interest.

\section*{Funding}
The second author was supported by a PhD scholarship from the Mathematics Division of Stellenbosch University.  The fourth author was partially supported by the 2024 ERC Synergy Grant ``DynAMiCs". 

\section*{Data Availability}
Data sharing is not applicable to this article as no datasets were generated or analyzed during the current study. 
All SageMath code used to perform the computations in this article is available at \url{https://github.com/hbatte/diophantine-ineq-2a3b-sage}.

\end{document}